\documentclass[hidelinks, 10pt]{amsart}

\usepackage{cite}

\usepackage[margin=3cm]{geometry}

\usepackage{amssymb}

\usepackage{hyperref}
\usepackage[all]{xy}
\usepackage{enumerate}

\usepackage{bbm}

\usepackage[dvipsnames]{xcolor}
\usepackage{graphicx}

\usepackage{tikz-cd}

\newcommand{\bC}{{\mathbb C}}

\newcommand{\bF}{{\mathbb F}}

\newcommand{\bP}{{\mathbb P}}
\newcommand{\bQ}{{\mathbb Q}}

\newcommand{\bZ}{{\mathbb Z}}

\newcommand{\obQ}{\overline{\mathbb{Q}}}

\newcommand{\cA}{{\mathcal A}}

\newcommand{\cF}{{\mathcal F}}

\newcommand{\cO}{{\mathcal O}}

\newcommand{\fA}{{\mathfrak A}}
\newcommand{\fB}{{\mathfrak B}}

\newcommand{\fE}{{\mathfrak E}}

\newcommand{\fT}{{\mathfrak T}}
\newcommand{\fU}{{\mathfrak U}}

\newcommand{\fX}{{\mathfrak X}}
\newcommand{\fY}{{\mathfrak Y}}
\newcommand{\fZ}{{\mathfrak Z}}

\newcommand{\fm}{{\mathfrak m}}

\newcommand{\fp}{{\mathfrak p}}

\newcommand{\et}{\text{et}}

\newcommand{\tF}{\widetilde{F}}

\newcommand{\nc}{\newcommand}

\nc\wh{\widehat}

\nc\on{\operatorname}

\nc\Gr{\on{Gr}}

\nc\Fl{\on{Fl}}

\newtheorem{thm}[subsection]{Theorem}

\DeclareMathOperator{\gr}{{gr}}

 \DeclareMathOperator{\Spf}{{Spf}}

\DeclareMathOperator{\Pic}{{Pic}}
\DeclareMathOperator{\rk}{{rk}}
\DeclareMathOperator{\ord}{{ord}}

\newcommand{\limto}{{\displaystyle\lim_{\longrightarrow}}}
\newcommand{\rightlim}{\mathop{\limto}}

\newcommand{\leftlim}{\mathop{\displaystyle\lim_{\longleftarrow}}}
\newcommand{\limfromn}{\leftlim\limits_{\raise3pt\hbox{$n$}}}
\newcommand{\limton}{\rightlim\limits_{\raise3pt\hbox{$n$}}}

\newcommand{\rightlimit}[1]{\mathop{\lim\limits_{\longrightarrow}}\limits%
                    _{\raise3pt\hbox{$\scriptstyle #1$}}}

\newcommand{\leftlimit}[1]{\mathop{\lim\limits_{\longleftarrow}}\limits%
                    _{\raise3pt\hbox{$\scriptstyle #1$}}}

\DeclareMathOperator{\Aut}{{Aut}}

\DeclareMathOperator{\im}{{Im}}

\DeclareMathOperator{\Res}{{Res}}

\newcommand{\ip}{\frac{1}{p}}

\makeatletter

\newcommand{\Rmnum}[1]{\expandafter\@slowromancap\romannumeral #1@}
\makeatother

\newtheorem*{theo}{Theorem}

\newtheorem{pr}{Proposition}[section]
\newtheorem{lm}[pr]{Lemma}

\newtheorem{cor}[pr]{Corollary}

\theoremstyle{definition}

\newtheorem{df}[pr]{Definition}
\newtheorem{rem}[pr]{Remark}

\newtheorem*{que}{Question}
\newtheorem{quest}{Question}

\numberwithin{equation}{section}

\newcommand{\Fr}{\operatorname{Fr}}

\newcommand{\dR}{\mathrm{dR}}
\newcommand{\cris}{\mathrm{cris}}

\begin{document}

\title[Rigid-analytic varieties with projective reduction violating Hodge symmetry]
{Rigid-analytic varieties with projective reduction violating Hodge symmetry}

\author[]{Alexander Petrov}
\address{Harvard University, USA}
\email{apetrov@math.harvard.edu}

\begin{abstract} 
We construct examples of smooth proper rigid-analytic varieties admitting formal model with projective special fiber and violating Hodge symmetry for cohomology in degrees $\geq 3$. This answers negatively the question raised by Hansen and Li in \cite{hl}.
 \end{abstract}

\maketitle

\section{Introduction}

A powerful tool in algebraic geometry over the field of complex numbers is viewing smooth projective varieties as complex manifolds and henceforth using complex analysis to study them. An important role for that technique is played by K\"ahler manifolds: this is a subclass of complex manifolds large enough to contain all smooth projective algebraic varieties and small enough to be amenable to Hodge theory. In particular, for any compact K\"ahler manifold $X$ the Hodge-to-de Rham spectral sequence $E_1^{i,j}=H^j(X,\Omega^i_X)\Rightarrow H^{i+j}_{\dR}(X/\bC)$ degenerates at the first page and the Hodge numbers satisfy Hodge symmetry: $\dim_{\bC} H^j(X,\Omega^i_X)=\dim_{\bC}H^i(X,\Omega^j_X)$.

The situation in $p$-adic geometry is somewhat different: $p$-adic Hodge theory has been developed for all smooth proper rigid-analytic varieties, without any ``K\"ahler" assumption. It was proven by Scholze \cite{schrig} that for any such variety the Hodge-to-de Rham spectral sequence degenerates. There are, however, examples of smooth proper rigid-analytic varieties that fail Hodge symmetry. This might suggest that there should be a natural narrower class of rigid-analytic varieties for which Hodge symmetry does hold.

 Let $K$ be a discretely valued $p$-adic field with ring of integers $\cO_K$ and perfect residue field $k$. In \cite{hl} David Hansen and Shizhang Li raised the following question (see also Conjecture 2.4 in \cite{schrio}) suggesting a candidate for that narrower class:

\begin{que}
Let $X$ be a smooth proper rigid-analytic variety over $K$ admitting a formal model over $\cO_K$ with {\it projective} special fiber. Is it true that $\dim_K H^i(X,\Omega^j_{X/K})=\dim_K H^j(X,\Omega^i_{X/K})$ for all $i, j$?
\end{que}

Hansen and Li answered this question positively for $i+j=1$ in Theorem 1.2 of \cite{hl}. The main result of this text is that the answer is in general negative for $i+j\geq 3$:

\begin{theo}[Corollary \ref{maincor}]\label{nonsym}
For every pair of positive integers $i\neq j$ with $i+j\geq 3$ there exists a smooth proper rigid-analytic variety $X$ over $\bQ_p$ admitting a smooth formal model $\fX$ over $\Spf \bZ_p$ with projective special fiber $\fX_{\bF_p}$ such that $$\dim_{\bQ_p} H^i(X,\Omega^j_{X/\bQ_p})\neq \dim_{\bQ_p} H^j(X,\Omega^i_{X/\bQ_p}).$$
\end{theo}

The idea behind the examples comes from the following difference between complex Hodge theory and $p$-adic Hodge theory. For any compact K\"ahler manifold $X$ not only there is an equality of numbers $h^{i,j}(X)=h^{j,i}(X)$ but there is a {\it canonical} isomorphism of $\bC$-vector spaces: $$H^{j}(X,\Omega^{i}_X)\simeq \overline{H^{i}(X, \Omega^{j}_X)}$$ In particular, for a finite group $G$ acting on $X$ the $G$-representations $H^{j}(X,\Omega^{i}_X)$ and $H^{i}(X,\Omega^{j}_X)$  are dual to each other.

We show that the analogous statement in $p$-adic geometry fails already for abeloid varieties. We construct a formal abelian scheme $\fA$ (necessarily a non-algebraizable one) with an action of a finite group $G$ of order prime to $p$ such that the $G$-representations $H^0(A,\Omega^1_{A})$ and $H^1(A,\cO)$ on the cohomology of the abeloid generic fiber $A$ of $\fA$ are not dual to each other.

 Moreover, for $i+j= 3$ we can arrange (Proposition \ref{formabineq}) that the dimensions of $G$-invariants on the spaces $H^i(A,\Omega_A^j)$ and $H^j(A,\Omega_A^i)$ are different (note that these $G$-representations are obtained from $H^0(A,\Omega^1_{A})$ and $H^1(A,\cO)$ by taking exterior powers and tensor products). Taking the direct product with an auxiliary smooth projective formal scheme to make the $G$-action free and taking the quotient by $G$ gives the desired example. The examples for $i+j>3$ are obtained simply by taking the direct product with an appropriate projective scheme to move the asymmetry to higher cohomology via the K\"unneth formula.

This strategy does not work for $i+j= 2$ for a good reason: any smooth proper rigid-analytic variety that admits a {\it smooth} (not necessarily with a projective special fiber) formal model satisfies Hodge symmetry on first and second cohomology:

\begin{pr}[Corollary \ref{cordeg2}]\label{introdeg2}
If a smooth proper rigid-analytic variety $X$ over $K$ admits a smooth formal model over $\cO_K$ then $\dim H^0(X,\Omega^i_{X/K})=\dim_K H^i(X,\cO)$ for $i=1,2$.
\end{pr}

Proposition \ref{introdeg2} is a consequence of a certain self-duality of Frobenius action on the crystalline cohomology of the special fiber and weak admissibility (in the sense of filtered $\varphi$-modules) of de Rham cohomology of $X$ coming from the crystalline comparison theorem. As mentioned above, Hansen and Li proved this statement by a different method for $i=1$ in the case of (arbitrarily singular) projective reduction. Piotr Achinger has independently obtained Proposition \ref{introdeg2} in the case of smooth projective reduction by the same method, see \cite{a} for a treatment of a number of situations where Hodge symmetry does follow from the existence of a model with projective special fiber, including some cases when reduction is singular.

{\bf Notation.} Let $p$ be a prime number fixed throughout the text. If $X$ is an object of any of the three types $\{$smooth proper rigid-analytic variety over $K$, smooth proper formal scheme over $\Spf\cO_K$, smooth proper algebraic variety over an arbitrary field $F\}$, denote by $h^{i,j}(X)$ the number $\dim_K H^j(X,\Omega^i_{X/K}),\rk_{\cO_K}H^j(X,\Omega^i_{X/\cO_K})$ or $\dim_F H^j(X,\Omega^i_{X/F})$ respectively. Denote also by $\delta^{i,j}(X)$ the number $h^{i,j}(X)-h^{j,i}(X)$ in any of the three situations. For the convenience of notation, we declare $h^{i,j}(X)=0$ if $i$ or $j$ is negative. Note that $h^{i,j}(X)=h^{i,j}(\fX)$ if $X$ is the generic fiber of a smooth proper formal scheme $\fX$ by Lemma \ref{gencoh}. We say that an object $X$ ``satisfies Hodge symmetry in degree $d$" if $h^{i,j}(X)=h^{j,i}(X)$ for all $i,j$ with $i+j=d$ and $X$ ``satisfies Hodge symmetry" if it does so for all $d$.

{\bf Acknowledgments.} I am grateful to Piotr Achinger, Shizhang Li, Vadim Vologodsky, Zijian Yao and Bogdan Zavyalov for their interest and several useful discussions and to Piotr Achinger, Borys Kadets and Zhiyu Zhang for comments on the drafts of this text. I am indebted to Vasily Rogov whose talk on the work \cite{rogov} motivated the examples presented here.


\section{Hodge symmetry in degree $2$}

For the duration of this section let $K$ be a discretely valued field of characteristic zero with perfect residue field $k$ of characteristic $p$. Denote by $K_0$ the subfield $W(k)[\ip]$, and let $\sigma$ be its automorphism induced by Frobenius action on $k$.

The goal of this section is to observe that results of $p$-adic Hodge theory impose the following relation between Hodge numbers of a rigid-analytic variety with good reduction. A similar idea has been used in \cite{fm} I.4.4 to prove Hodge symmetry for smooth proper algebraic varieties over $K$.

\begin{pr}\label{deg2}
If $\fX$ is a smooth proper formal scheme over $\cO_K$ then for its generic fiber $X$ and every $n$ the following relation among $h^{i,j}:=h^{i,j}(X)$ holds:
\begin{equation}\label{deg 2: main}
h^{1,n-1}+2h^{2,n-2}+\dots+n \cdot h^{n,0}=n\cdot h^{0,n}+(n-1)\cdot h^{1,n-1}+\dots + h^{n-1,1}.
\end{equation}
\end{pr}

\begin{cor}\label{cordeg2}
For a smooth proper formal scheme $\fX$ over $\Spf\cO_K$ we have the following for the generic fiber $X=\fX_K$

(i) $h^{1,0}(X)=h^{0,1}(X)$ and $h^{2,0}(X)=h^{0,2}(X)$

(ii) for odd $n$ the Betti numbers $\dim_{\bQ_p}H^n_{\acute{e}t}(X_{\overline{K}},\bQ_p)=\dim_{K}H^n_{\dR}(X/K)$ are even. 
\end{cor}

\begin{proof}
(i) The equality \ref{deg 2: main} takes forms $h^{1,0}=h^{0,1}$ and $h^{1,1}+2h^{2,0}=2h^{0,2}+h^{1,1}$ for $n=1,2$ respectively.

(ii) The sum of the left-hand side and the right-hand sied of \ref{deg 2: main} is, of course, even and is also equal to $n\cdot (h^{0,n}+h^{1,n-1}+\dots +h^{n,0})=n\cdot \dim_{K}H^n_{\dR}(X/K)$ so $\dim_K H^{n}_{\dR}(X/K)$ is even for odd $n$.
\end{proof}

We first recall the setup of rational $p$-adic Hodge theory. Let $\mathrm{MF}^{\varphi}_{K}$ be the category of filtered $\varphi$-modules over $K$ \cite{foncris}. Its objects are finite-dimensional vector spaces $D$ over $K_0$ equipped with a semi-linear automorphism $\varphi_D:D\to D$ and a deceasing filtration $F^iD_K$ on the $K$-vector space $D_K:=D\otimes_{K_0}K$. To every such object we can attach two integers 
\begin{equation}
t_N(D)=\sum\limits_{\alpha\in\bQ}\alpha\cdot \dim_{K_0}D_{\alpha}\text{ and }t_H(D)=\sum\limits_{i\in\bZ}i\cdot \dim_{K}\gr^i D_K
\end{equation}

Here $D_{\alpha}$ is the slope $\alpha$ direct summand of the $\varphi$-module $D$ under the Dieudonne-Manin decomposition and $\gr^iD_K:=F^iD_K/F^{i+1}D_K$ are the graded pieces of the filtration. Note that $t_N$ depends only on the semi-linear endomorphism and we will also use the notation $t_N(D)$ for a $\varphi$-module $D$ not equipped with a filtration. A filtered $\varphi$-module $D\in \mathrm{MF}^{\varphi}_{K}$ is called {\it weakly admissible} if $t_N(D)=t_H(D)$ and for every subobject $D'\subset D$ we have $t_N(D')\geq t_H(D')$. For an integer $i$ denote by $K_0(i)$ the filtered $\varphi$-module given by a one-dimensional vector space $D=K_0$ with $\varphi_d=p^{-i}\sigma$ and the filtration defined by $F^{-i}D_K=D_K,F^{-i+1}D_K=0$. For a filtered $\varphi$-module $D$ define its $i$-fold twist $D(i)$ as the tensor product $D\otimes_{K_0}K_0(i)$, we use the same notation for twisting $\varphi$-modules not equipped with a filtration.

We have the following symmetry of Frobenius slopes, due to \cite{suh}:

\begin{lm}[Corollary 2.2.4 in \cite{suh}]\label{slopesym}Let $Y$ be a smooth proper variety over $k$. For the $\varphi$-module $D=H^n_{\cris}(Y/W(k))[\ip]$ we have $t_N(D)=n\cdot\dim_{K_0}D-t_N(D)$. 
\end{lm}
\begin{proof}
For the convenience of the reader we recall the proof in the case when $Y$ is projective over an arbitrary perfect field or arbitrary proper defined over a finite field.

If $Y$ is projective, denote by $L\in H^2_{\cris}(Y/W(k))$ the crystalline Chern class of an ample line bundle on $Y$. The hard Lefschetz theorem for crystalline cohomology \cite{km} shows that cup-product with $L^{d-\min(n,2d-n)}$ induces an isomorphism $H^n_{\cris}(Y/W(k))[\ip](n-d)\simeq H^{2d-n}_{\cris}(Y/W(k))[\ip]$ of $\varphi$-modules over $K_0$. On the other hand, Poincare duality provides an isomorphism $H^{2d-n}_{\cris}(Y/W(k))[\ip]\simeq H^n_{\cris}(Y/W(k))^{\vee}[\ip](-d)$. Combining these isomorphisms we get $H^n_{cris}(Y/W(k))[\ip](n)\simeq H^n_{\cris}(Y/W(k))[\ip]^{\vee}$. Since $t_N(D^{\vee}(-n))=-t_N(D(n))=-t_N(D)+n\cdot \dim_{K_0}D$ we get the desired equality.

If $Y$ is not projective but is defined over a finite field $k=\bF_{q}$ the result follows from Weil conjectures: the multiset of $q$-power Frobenius eigenvalues consists of algebraic integers and is stable under complex conjugation. Hence, it is stable under the operation $\alpha\mapsto q^n\alpha^{-1}=\overline{\alpha}$ so the multiset of slopes is stable under $\lambda\mapsto n-\lambda$, as desired.
\end{proof}

\begin{proof}[Proof of Proposition \ref{deg2}]
Consider the filtered $\varphi$-module given by $D=H^n_{\cris}(\fX_k/W(k))[\ip]$ with the Frobenius structure induced by $\Fr:\fX_k\to\fX_k\times_{k,\sigma}k$ and the filtration on $D_K$ is given by the Hodge filtration on the de Rham cohomology of the generic fiber $H^n_{\dR}(X/K)\simeq H^n_{\cris}(\fX/W(k))\otimes_{W(k)}K=D_K$. 

By the crystalline comparison theorem proved in \cite{bms} Theorem 1.1(i) for smooth proper formal schemes, the filtered $\varphi$-module $D$ is obtained by applying the functor $D_{\cris}$ to the $p$-adic Galois representation $H^n_{\et}(X_{\overline{K}},\bQ_p)$. In particular, $D$ is weakly admissible by Proposition 4.4.5 \cite{foncris}. 

Hence, $t_H(D)=t_N(D)$ so by Lemma \ref{slopesym} we get $t_H(D)=n\cdot \dim_K H^n_{\dR}(X/K)-t_H(D)$. As $t_H(D)=h^{1,{n-1}}+2h^{2,n-2}+\dots + n\cdot h^{n,0}$ and $\dim_K H^n_{\dR}(X/K)=h^{0,n}+h^{1,n-1}+\dots + h^{n,0}$ we get the desired equality.
\end{proof}

\begin{rem}\label{admissfutile}
The above proof used only that the endpoints of the Newton and Hodge polygons of $D$ coincide. The full strength of weak admissibility, however, does not put any further restrictions on the Hodge numbers. For any tuple of non-negative integers $h^{0,n},h^{1,n-1},\dots, h^{n,0}$ satisfying (\ref{deg 2: main}) equip the vector space $D:=\bQ_p^b$ of dimension $b=h^{0,n}+h^{1,n-1}+\dots + h^{n,0}$ with an endomorphism $\varphi$ that has irreducible characteristic polynomial with all roots having valutation $a/b$ where $a=h^{1,n-1}+2h^{2,n-2}\dots + n\cdot h^{n,0}$. Endowing $D$ with an arbitrary filtration $F^iD$ such that $\dim_{\bQ_p}F^iD=h^{i,n-i}+h^{i+1,n-i-1}+\dots+h^{n,0}$ turns $D$ into an object of $\mathrm{MF}_{\bQ_p}^{\varphi}$ with Hodge numbers $h^{i,n-i}$ which is weakly admissible just because there are no proper non-zero submodules stable under $\varphi$. Alternatively, we can apply Theorem 1 of \cite{fr} that shows the existence of a weakly admissible module over $W(\overline{k})[\ip]$ with any given Newton polygon and any Hodge polygon having the same endpoints and lying below it. 
\end{rem}

\section{Cyclic group acting on a formal abelian scheme}

Fix once and for all a prime number $l$ satisfying the following condition: \begin{equation}\label{lcondition}l\neq p\textit{ and the order of }p\textit{ in the multiplicative group }(\bZ/l)^{\times}\textit{ is divisible by }4\end{equation} This condition is equivalent to $l$ being a prime divisor of a number of the form $p^{2r}+1$. Denote by $k$ the finite field $\bF_{p}(\mu_l)$ obtained by adjoining to $\bF_p$ all $l$-th roots of unity and by $\kappa\subset k$ the finite field $\kappa=\bF_{p^2}$. Let $G$ denote the cyclic group $\bZ/l$ with a chosen generator $\sigma \in G$. We will denote by $\bZ[\mu_l]$ the ring of integers of the cyclotomic extension $\bQ(\mu_l)$ inside a chosen algebraic closure $\obQ$.

The goal of this section is to exploit the existence of abelian varieties over a finite field with multiplication by $\bZ[\mu_l]$ such that the resulting CM type is different from those appearing over characteristic zero fields. This allows us to construct a formal abelian scheme with an action of $G$ having asymmetric dimensions of invariant spaces on Hodge cohomology groups in degree $3$:

\begin{pr}\label{formabineq}
There exists a formal abelian scheme $\fZ$ with an action of $G$ over $\Spf W(k')$ for some finite field $k'$ such that \begin{equation}\label{formalabineq:noneq}\rk H^0(\fZ,\Omega^3_{\fZ/W(k')})^G< \rk H^3(\fZ, \cO)^G.\end{equation}
\end{pr}

\begin{rem}
The proof also produces examples with inequalities going in the other direction, we record here this particular version for the convenience of applying it to Lemma \ref{specialfibersym}.
\end{rem}

We first collect the necessary facts about abelian varieties in Lemmas \ref{modpabvar}-\ref{geomtypical} and then use them together with the combinatorial Lemma \ref{choosetypical} to prove the Proposition \ref{formabineq} at the end of this section. All the facts about abelian varieties used in this section are contained in \cite{cco} (especially, cf. Example 4.1.2 there). We include here the proofs in an attempt to make the exposition relatively self-contained.

\begin{lm}\label{modpabvar}
There exists an abelian variety $A$ over $\kappa$ of dimension $\frac{l-1}{2}$ equipped with an action of $G$ such that the eigenvalues of $\sigma\in G$ on $H^1_{\cris}(A/W(\kappa))\otimes_{W(\kappa)}W(k)[\ip]$ are the $l-1$ pairwise different nontrivial roots of unity of order $l$ and the $p^2$-Frobenius endomorphism of $A$ is given by $\sigma\circ[p]_A$.
\end{lm}

\begin{proof}
Let $\zeta_l$ be a primitive $l$-th root of unity in $\obQ$. By Honda--Tate theory, there exists an abelian variety $A'$ (unique up to isogeny) over $\kappa$ with eigenvalues of the $p^2$-Frobenius endomorphism on first etale cohomology given by the conjugates of the Weil number $p\cdot\zeta_l$. Since $[\bQ_p(\mu_l):\bQ_p]=\ord_{(\bZ/l)^{\times}}p$ is even we have $2\dim A'=[\bQ(\mu_l):\bQ]=l-1$ by Theoreme 1(ii) of \cite{t}. If $\varphi_{A'}$ denotes the $p^2$-Frobenius endomorphism of $A'$ then $\varphi_{A'}^l$ acts by $p^l$ on $H^1_{\et}(A'_{\overline{\bF}_p},\bZ_t)$ (where $t$ is any prime different from $p$) so $\varphi_{A'}^l$ is equal to the multiplication-by-$p^l$ endomorphism of $A'$. Hence, $\varphi_{A'}$ induces the action of the subalgebra $\bZ[p\zeta_l]\subset \bZ[\mu_l]$ on the abelian variety $A'$ with $p\zeta_l$ acting by $\varphi_{A'}$.

Proposition 1.7.4.4 of \cite{cco} shows that the Serre's tensor product $A:=\bZ[\zeta_l]\otimes_{\bZ[p\zeta_l]}A'$ is an abelian variety isogenous to $A'$ that has an action of $G$ with $\sigma$ acting as $\zeta_l$. The eigenvalues of $\sigma$ on $H^1_{\cris}(A/W(\kappa))$ are equal to the eigenvalues of $\varphi_A$ divided by $p$ and, since characteristic polynomials of $\varphi_A$ on crystalline and etale cohomology are equal, the eiegnvalues of $\sigma$ are the $l-1$ Galois conjugates of $\zeta_l$, as desired.
\end{proof}

\begin{rem}
The variety $A'$ can be constructed as the quotient $(\Res_{\bF_{p^2}}^{\bF_{p^{2l}}}(E\times_{\bF_{p^2}}\bF_{p^{2l}}))/E$ for a supersingular elliptic curve $E$ over $\bF_{p^2}$ with the trace of Frobenius endomorphism equal to $2p$.
\end{rem}


The map $g\mapsto g^p$ is an automorphism of the group $G$. For a representation $\rho:G\to \Aut(V)$ on a module $V$ we denote by $V^{\tau}$ the representation on the same module in which $g\in G$ acts by $\rho(g^p)$.
\begin{lm}\label{ablift}
For any $A$ as in Lemma \ref{modpabvar} we have the following:

(i) the representations $H^1(A,\cO)$ and $H^0(A,\Omega^1_{A/\kappa})^{\tau}$ are isomorphic.

(ii) $A$ lifts to a formal abelian scheme $\fA$ over $\Spf W(\kappa)$ together with an action of $G$ such that the representations $H^1(\fA,\cO)$ and $H^0(\fA,\Omega^1_{\fA/W(\kappa)})^{\tau}$ are isomorphic.
\end{lm}

\begin{proof}
\newcommand{\kp}{\kappa}
(i) Denote by $A^{(1)}=A\times_{\kappa,\Fr}\kappa$ the Frobenius-twist of $A$. For a vector space $W$ over $\kappa$ denote as well by $W^{(1)}:=W\otimes_{\kappa,\Fr}\kappa$ its twist. Note that the double twist $A^{(2)}:=(A^{(1)})^{(1)}$ is canonically identified with $A$.

The de Rham cohomology $H^1_{\dR}(A/\kappa)$ carries the Hodge filtration $H^1_{\dR}(A/\kp)=F^0\supset F^1\supset F^2=0$ and the conjugate filtration $H^1_{\dR}(A/\kp)=G_2\supset G_1\supset G_0=0$ satisfying $F^0/F^1=H^0(A,\Omega^1_{A/\kp}),F^1/F^2=H^1(A,\cO),G_2/G_1=H^0(A^{(1)},\Omega^1_{A^{(1)}/\kp}),G_1/G_0=H^1(A^{(1)},\cO)$, see \S 7 of \cite{k}. The compositions $$H^1(A^{(1)},\cO)=G_1\to H^1_{\dR}(A/\kp)\to F^0/F^1=H^1(A,\cO)$$ and $$H^1_{\dR}(A^{(1)}/\kappa)=H^1_{\dR}(A/\kp)^{(1)}\to (F^0/F^1)^{(1)} = H^1(A,\cO)^{(1)}=G_1/G_0\to H^1_{\dR}(A/\kp)$$ are both induced by the relative Frobenius morphism $\Fr_A:A\to A^{(1)}$. 

Since $\Fr_A^2$ is equal to the multiplication by $p$ endomorphism $[p]_A$ up to composing with an automorphism, $\Fr_A^2$ induces the zero map on $H^1_{\dR}(A/k)$. Equivalently, $\ker(\Fr_A^*:H^1_{\dR}(A^{(1)}/\kp)\to H^1_{\dR}(A/\kp))$ contains $\im(\Fr_A^*:H^1_{\dR}(A/\kp)=H^1_{\dR}(A^{(2)}/\kp)\to H^1_{\dR}(A^{(1)}/\kp))=G_1^{(1)}$. The sum of the dimensions of these two vector space is $\dim_{\kp} H^1_{\dR}(A/\kp)=2g$ and dimension of the image is equal to $g$, so the containment is in fact equality and we get $G_1=F^1$. 

This gives a $G$-equivariant $\kp$-linear isomorphism $H^1(A^{(1)},\cO)\simeq H^0(A,\Omega^1_{A/\kp})$. If $V$ is a representation of $G$ on a vector space over a field of characteristic $p$ then the  character of Frobenius-twisted representation $V^{(1)}$ is  equal to that of $V^{\tau}$ so the first assertion of the lemma follows.

(ii) By Theorem V.1.10 of \cite{m} the category of formal abelian schemes over $W(\kp)$ is equivalent to the category of pairs $(A_0, \tF)$ where $A_0$ is an abelian variety over $\kp$ and $\tF$ is a $W(\kp)$-submodule of $H^1_{\cris}(A_0/W(\kp))$ such that the quotient $H^1_{\cris}(A_0/W(\kp))/\tF$ is torsion-free and $\tF/p$ is equal to the first step of Hodge filtration $F^1\subset H^1_{\dR}(A/\kp)=H^1_{cris}(A/W(\kp))/p$. 

Since the eigenvalues of $\sigma$ on $H^1_{\cris}(A/W(\kp))$ are pair-wise different, the same is true for the action of $\sigma$ on $H^1_{\dR}(A/\kp)$ so there is a unique isomorphism $H^1_{\dR}(A/\kp)\simeq H^0(A,\Omega^1_{A/\kp})\oplus H^1(A,\cO)$ of $G$-representations splitting the Hodge filtration on $H^1_{\dR}(A/\kp)$. Since the isomorphism class of a representation of $G$ on a finite free $W(\kp)$-module is completely determined by its reduction modulo $p$, there is a unique $G$-equivariant decomposition $H^1_{\cris}(A/W(\kp))=\tF\oplus\tF'$ lifting the decomposition of the de Rham cohomology. The formal lift $\fA$ of $A$ defined by $\tF$ then comes equipped with a lift of the action of $G$ because $\tF\subset H^1_{\cris}(A/W(\kp))$ is stable under $G$. The assertion about the action of $G$ on Hodge cohomology groups of $\fA$ is a formal consequence of (i).
\end{proof}

The result of Lemma \ref{ablift}(i) is specific to positive characteristic: for an abelian variety $B$ with an action of $G$ over a field $F$ of characteristic zero the representations $H^0(B,\Omega^1_{B/F})$ and $H^1(B,\cO)^{\vee}$ have to be isomorphic by the next lemma, whereas for $A$ as above the representations $H^0(A,\Omega^1_{A/\kappa})$ and $H^1(A,\cO)^{\vee}\simeq H^0(A,\Omega^1_{A/\kappa})^{\tau \vee}$ are self-dual and are not isomorphic to each other (in fact, any representation of $G$ on a $\bF_{p^2}$-vector space has to be self-dual because $-1=p^{2r}$ in $(\bZ/l)^{\times}$ for some $r$ by the assumption on $l$). The same is happening for the lift $\fA$, so, in particular, it has to be non-algebraizable.

\begin{lm}\label{char0dual}
For an abelian variety $B$ over a field $F$ of characteristic zero equipped with an action of a finite group $H$ the representations $H^0(B,\Omega^1_{B/F})$ and $H^1(B,\cO)^{\vee}$ are isomorphic. 
\end{lm}

\begin{proof}
Choose an ample line bundle $L$ on $B$. Then $M:=\bigotimes\limits_{h\in H}h^*(L)$ is an ample line bundle inducing a $H$-equivariant polarization $\lambda_{M}:B\to B^{\vee}$ that is separable because $\mathrm{char}\, F=0$. The induced map $\lambda_{M}^*:H^0(B^{\vee},\Omega^1_{B^{\vee}/F})\to H^0(B,\Omega^1_{B/F})$ provides a $H$-equivariant isomorphism between $H^0(B,\Omega^1_{B/F})$ and $H^1(B,\cO)^{\vee}\simeq H^0(B^{\vee},\Omega^1_{B^{\vee}/F})$
\end{proof}

The abelian schemes constructed in Lemma \ref{ablift} do not yet provide the desired $\fZ$ because for all $i,j$ the module $H^i(\fA,\Omega^j_{\fA/W(\kappa)})$ is the $\tau$-twist of $H^j(\fA,\Omega^i_{\fA/W(\kappa)})$ and, in particular, these modules have equal ranks of invariants. We will break this symmetry by taking the direct product with an appropriate algebraic abelian scheme with complex multiplication by $\bQ(\mu_l)$.

Recall that $k$ is an extension of $\bF_p$ containing all $l$-th roots of unity, so any representation of $G$ on a finite free module over $W(k)$ is isomorphic to a direct sum of characters $\chi_{\zeta}$ given by sending $\sigma$ to a degree $l$ root of unity $\zeta\in \mu_l\subset W(k)^{\times}$. 

\begin{df}We call a representation $U$ of $G$ on a finite free $W(k)$-module {\it typical} if it has the form $\bigoplus\limits_{i=1}^d\bigoplus\limits_{\zeta\in S_i}\chi_{\zeta}$ where each $S_i$ is a subset of $\mu_l$ of cardinality $\frac{l-1}{2}$ such that $S_i\cap S_i^{-1}=\varnothing$ and $d$ is some number.
\end{df} 

If $\fB$ is an algebraic abelian scheme over $W(k)$ of dimension $\frac{l-1}{2}$ with a non-trivial action of $G$ then the representation $H^0(\fB,\Omega^1_{\fB/W(k)})$ is typical: for the base change of $\fB$ to $\bC$ singular cohomology $H^1(\fB_{\bC},\bZ)$ is a free abelian group of rank $l-1$ with a non-trivial action of $G$. Hence, all non-trivial characters appear in $H^1(\fB_{\bC},\bZ)$ exactly once and by Lemma \ref{char0dual} these characters have to be distributed between $H^0(\fB_{\bC},\Omega^1_{\fB_{\bC}/\bC})=H^0(\fB,\Omega^1_{\fB/W(k)})\otimes_{W(k)}\bC$ and $H^1(\fB_{\bC},\cO)$ in such a way that $\chi_{\zeta}$ and $\chi_{\zeta^{-1}}$ never occur in the same piece of the Hodge decomposition. The next lemma deduces from the theory of complex multiplication that every typical representation arises from an abelian scheme over the ring of integers in an unramified extension of $\bQ_p$:

\begin{lm}\label{geomtypical}
For any {\it typical} representation of $G$ on a finite free $W(k)$-module $U$ there exists an abelian scheme $\fB$ with an action of $G$ over the ring of integers $\cO_L=W(k')$ of a finite unframified extension $L=W(k')[\ip]\supset W(k)[\ip]$ such that $H^0(\fB,\Omega^1_{\fB/\cO_L})\simeq U\otimes_{W(k)}{\cO_L}$ and $H^1(\fB,\cO)\simeq U^{\vee}\otimes_{W(k)}\cO_L$ as $G$-representations.
\end{lm}

\begin{proof}
It is enough to prove the lemma for typical representations of rank $\frac{l-1}{2}$ because we have $H^0(\fB_1\times_{\cO_L}\fB_2,\Omega^1_{\fB_1\times_{\cO_L}\fB_2/\cO_L})\simeq H^0(\fB_1,\Omega^1_{\fB_1/\cO_L})\oplus H^0(\fB_2,\Omega^1_{\fB_2/\cO_L})$ and likewise $H^1(\fB_1\times_{\cO_L}\fB_2,\cO)\simeq H^1(\fB_1,\cO)\oplus H^1(\fB_2,\cO)$ for all abelian schemes $\fB_1,\fB_2$, so the general case can be obtained by taking the products.

Fix a non-trivial root of unity $\zeta_l\in \bQ(\mu_l)$ and an embedding $\bQ(\mu_l)\subset W(k)[\ip]$, giving a bijection between roots of unity in $\bQ(\mu_l)$ and $W(k)$. Each non-trivial root of unity $\zeta\in\mu_l(W(k))$ thus defines an embedding of the cyclotomic field $\varphi_{\zeta}:\bQ(\mu_l)\to \overline{\bQ}$ into the algebraic closure given by $\zeta_l\mapsto \zeta$. 

Let $U$ be a typical representation of rank $\frac{l-1}{2}$ over $W(k)$. We associate to $U$ the set $\Phi=\{\varphi_{\zeta}|\chi_{\zeta^{-1}}\text{ appears in }U\}$ of embeddings. By the definition of typical representation $\Phi$ is a CM type of the field $\bQ(\mu_l)$. Since $\bQ(\mu_l)$ is Galois over $\bQ$, the reflex field of $\Phi$ is contained in $\bQ(\mu_l)$. By \cite{cco} Corollary A.4.6.5 applied with $N=p$ (beware that the symbol "$p$" bears its own meaning in this corollary) there exists an abelian variety $B'$ with complex multiplication by $\bQ(\mu_l)$ of type $\Phi$ defined over an unramified extension $F$ of the reflex field that has good reduction at all primes above $p$. Take $L=F_{\fp}$ for some prime $\fp\subset\cO_F$ above $p$ and consider the good model $\fB'$ of $B'$ over $\cO_L$. The isogeny action of $\bQ(\mu_l)$ then extends onto $\fB'$. In particular, embedding $G$ into $\bQ(\mu_l)$ via $\sigma\mapsto\zeta_l$ gives an action of $G$ on $\fB'$ in the isogeny category. By our choice of $\Phi$ the $G$-representation $H^0(\fB',\Omega^1_{\fB'/\cO_L})\otimes_{\cO_L}L$ is isomorphic to $U\otimes_{W(k)}L$.

Analogously to the proof of Lemma \ref{modpabvar}, Serre's tensor product construction gives an isogeny $\fB'\to \fB$ to an abelian scheme $\fB$ that carries a genuine action of $\bZ[\mu_l]$. In particular, $G$ acts on $\fB$ such that $H^0(\fB,\Omega^1_{\fB/\cO_L})$ is isomorphic to $U\otimes_{W(k)}\cO_L$, as desired. To see that $H^1(\fB,\cO)$ is isomorphic to $U^{\vee}\otimes_{W(k)}\cO_L$ it is enough to check the analogous statement for the generic fiber $\fB_L$ and this is given by Lemma \ref{char0dual}.
\end{proof}

\begin{rem}
We have appealed to stronger results of \cite{cco} to make sure that there exists $\fB$ defined over the ring of integers in an {\it unramified} extension of $\bQ_p$. If we did not care about the field of definition of the ultimate counterexamples to Hodge symmetry we could have applied a cruder form of CM theory given e.g. by \cite{t} Lemme 4.
\end{rem}

\begin{lm}\label{choosetypical}
Let $V$ be a representation of $G$ on a finite free $W(k)$-module such that $V$ is isomorphic to its dual $V^{\vee}$ but is not isomorphic to the twist $V^{\tau}$. Then there exists a typical representation $U$ such that \begin{equation}\label{choosetypical:noneq}\rk \Lambda^3(V\oplus U)^G\neq \rk\Lambda^3(V^{\tau}\oplus U^{\vee})^G\end{equation}
\end{lm}

\begin{proof}
Suppose on the contrary that these ranks of invariant subspaces are equal for all typical $U$. Using that $\Lambda^3(V\oplus U)\simeq \Lambda^3 V\oplus\Lambda^2V\otimes U\oplus V\otimes\Lambda^2 U\oplus\Lambda^3 U$ and the fact that twisting by an automorphism of $G$ does not change the invariant submodule of a representation we get $$\rk (\Lambda^2V\otimes U^{\oplus r})^G +\rk (V\otimes\Lambda^2(U^{\oplus r}))^G = \rk(\Lambda^2 V^{\tau}\otimes U^{\vee\oplus r})^G+\rk(V^{\tau}\otimes\Lambda^2 (U^{\vee\oplus r}))^G$$ for all typical $U$ and any multiplicity $r$ (by definition, the class of typical representations is closed under direct sums). Since $\Lambda^2(U^{\oplus r})\simeq (\Lambda^2U)^{\oplus r}\oplus(U^{\otimes 2})^{\oplus\binom{r}{2}}$ both sides of the equality are quadratic polynomials in $r$ for a fixed $U$ and comparing the leading coefficients we get $$\rk(V\otimes U^{\otimes 2})^G=\rk(V^{\tau}\otimes U^{\vee\otimes 2})^G$$ As $V$ is assumed to be self-dual this equality is equivalent to $\rk(V\otimes U^{\otimes 2})^G=\rk(V^{\tau}\otimes U^{\otimes 2})^G$. In other words, the difference of characters $\chi_V-\chi_{V^{\tau}}$ is orthogonal to any character $\chi_{U^{\otimes 2}}$ for a typical $U$ with respect to the pairing $\langle\chi_1,\chi_2\rangle:=\sum\limits_{g\in G}\chi_1(g)\chi_2(g)$ on the character ring $W(k)[G]$. Moreover, this implies that $\chi_V-\chi_{V^{\tau}}$ is orthogonal to $\chi_{U_1\otimes U_2}=\chi_{U_1}\cdot \chi_{U_2}$ for any two typical representations $U_1,U_2$ because $2\chi_{U_1\otimes U_2}=\chi_{(U_1\oplus U_2)^{\otimes 2}}-\chi_{U_1^{\otimes 2}}-\chi_{U_2^{\otimes 2}}$.

If $\zeta$ is any non-trivial root of unity, then for a set $\zeta_1,\dots,\zeta_{\frac{l-3}{2}}$ of pairwise different roots such that $\zeta_i\neq \zeta_j^{-1}$ and $\zeta_i\neq\zeta^{\pm 1}$ for all $i,j$ both representations $\chi_{\zeta}\oplus\chi_{\zeta_1}\oplus\dots\oplus \chi_{\zeta_{\frac{l-3}{2}}}$ and $\chi_{\zeta^{-1}}\oplus\chi_{\zeta_1}\oplus\dots\oplus \chi_{\zeta_{\frac{l-3}{2}}}$  are typical. Hence the span of characters of typical representations in $W(k)[G]$ contains all characters of the form $\chi_{\zeta}-\chi_{\zeta^{-1}}$ for a non-trivial root of unity $\zeta$. By linearity of the pairing, $\chi_V-\chi_{V^{\tau}}$ is hence also orthogonal to $\chi^2_{\zeta}+\chi^2_{\zeta^{-1}}=(\chi_{\zeta}-\chi_{\zeta^{-1}})^2+2\chi_{1}$. However, $\langle \chi_V,\chi^2_{\zeta}+\chi^2_{\zeta^{-1}}\rangle=2\langle\chi_V,\chi^2_{\zeta}\rangle$ by self-duality of $V$ and likewise for $V^{\tau}$ so we get a contradiction since, by assumption, there exists $\zeta$ such that multiplicities of $\chi_{\zeta^{-2}}$ in $V$ and $V^{\tau}$ are different.
\end{proof}

\begin{proof}[Proof of Proposition \ref{formabineq}]
Let $\fA$ be a formal abelian scheme with an action of $G$ provided by Lemma \ref{ablift}(ii). Consider the representation $V=H^0(\fA,\Omega^1_{\fA/W(\kappa)})$. We have $G$-equivariant isomorphisms $V\oplus V^{\tau}\simeq H^0(\fA,\Omega^1_{\fA/W(\kappa)})\oplus H^1(\fA,\cO)\simeq H^1_{\dR}(\fA/W(\kappa))\simeq H^1_{\cris}(A/W(\kappa))$, so every non-trivial character $\chi_{\zeta}$ appears in exactly one of $V$ and $V^{\tau}$ with multiplicity one. It follows that $V\simeq (V^{\tau})^{\tau}$. Since $V^{\vee}$ is obtained from $V$ by applying the twist $\tau$ an even number of times, the representation $V\otimes_{W(\kappa)}W(k)$ satisfies the assumptions of Lemma \ref{choosetypical} and there exists a typical representation $U$ such that non-equality (\ref{choosetypical:noneq}) holds. Let $\fB$ be an abelian scheme over $\cO_L=W(k')$ provided by Lemma \ref{geomtypical} applied to $U$. The product $\fZ'=\cA\times_{W(k)}\fB$ equipped with the diagonal action then has $H^0(\fZ',\Omega^1_{\fZ/\cO_L})\simeq (V\oplus U)\otimes_{W(k)}{\cO_L}$ and $H^1(\fZ',\cO)\simeq (V^{\tau}\oplus U^{\vee})\otimes_{W(k)}{\cO_L}$ as $G$-representations.

Since $H^0(\fZ',\Omega^3_{\fZ/\cO_L})\simeq \Lambda^3 H^0(\fZ',\Omega^1_{\fZ'/\cO_L})$ and $H^3(\fZ',\cO)\simeq \Lambda^3 H^1(\fZ',\cO)$ we get $\rk H^0(\fZ',\Omega^3_{\fZ/\cO_L})^G\neq \rk H^3(\fZ',\cO)^G$. If $\rk H^0(\fZ',\Omega^3_{\fZ/\cO_L})^G> \rk H^3(\fZ',\cO)^G$ then $\fZ=\fZ'$ gives the desired inequality \ref{formalabineq:noneq}. If the inequality goes the other way, take $\fZ=\fZ'^{\vee}$ using that there are $G$-equivariant isomorphisms $H^0(\fZ'^{\vee},\Omega^1_{\cO_L})\simeq H^1(\fZ',\cO)^{\vee}$ and $H^1(\fZ'^{\vee},\cO)\simeq H^0(\fZ',\Omega^1_{\fZ'/\cO_L})^{\vee}$.
\end{proof}

\begin{rem}
(i) The same argument works just as well to provide an abelian scheme with different ranks of invariants $\rk H^i(\fZ,\Omega^j_{\fZ/W(k)})^G\neq \rk H^i(\fZ,\Omega^j_{\fZ/W(k)})^G$ for any $i\neq j, i+j\geq 3$. We chose to treat here only the case $(i,j)=(3,0)$ and obtain counterexamples to symmetry in higher degrees by taking products with auxiliary varieties because that simplifies the argument in the next section.

(ii) I am not aware of a simple way of choosing $\fZ$ that would work uniformly for all $p$. However, for any fixed $l$ everything can be made more explicit for primes $p$ for which $l$ satisfied condition (\ref{lcondition}). For instance, for $p$ congruent to $2$ or $3$ modulo $5$ we can take $l=5$ so that $\dim\fA=2$ and $\fB$ can also be chosen to be $2$-dimensional. In this case (for any of the particular choices of $\fA$ and $\fB$ that give non-symmetric dimensions of invariants) $H^3_{\dR}(A\times B/L)^G$ is a filtered $\varphi$-module with Hodge numbers $(h^{3,0},h^{2,1},h^{1,2},h^{0,3})=(0,5,2,1)$ or $(1,2,5,0)$ and single slope $\frac{3}{2}$ (the slopes can be computed by the Shimura-Taniyama formula). Note that such Hodge numbers satisfy the relation (\ref{deg 2: main})\end{rem}

\section{Main example}

\begin{thm}\label{main}
For every pair of distinct positive integers $i\neq j$ with $i+j\geq 3$ there exists a smooth proper formal scheme $\fX$ over $\bZ_p$ with projective special fiber $\fX_{\bF_p}$ such that the Hodge cohomology groups $H^i(\fX,\Omega^j_{\fX/\bZ_p})$ and $H^j(\fX,\Omega^i_{\fX/\bZ_p})$ are free $\bZ_p$-modules of different rank.
\end{thm}

We first record two standard general facts from formal geometry. The main reference on rigid-analytic generic fibers of formal schemes is \cite{bl}. Let $K$ be a discretely valued field of characteristic zero with ring of integers $\cO_K$, maximal ideal $\fm\subset\cO_K$ and perfect residue field $k=\cO_K/\fm$ of characteristic $p$. 

\begin{lm}\label{gencoh}
If $\fX$ is a smooth proper formal scheme over $\Spf\cO_K$ with generic fiber $X$, then there is a canonical isomophism $H^i(X,\Omega^j_{X/K})\simeq H^i(\fX,\Omega^j_{\fX/\cO_K})[\ip]$. In particular, $h^{i,j}(X)=h^{i,j}(\fX)$.
\end{lm}

\begin{proof}
Let $(\fU_i)_{i\in I}$ be a finite cover of $\fX$ by affine open formal subschemes. Their generic fibers $U_i$ induce a covering of $X$ by affinoid subdomains. For any coherent sheaf $\cF$ on $\fX$ its cohomology is computed by the \v{C}ech complex $\bigoplus_{i\in I}\cF(\fU_i)\to \bigoplus_{i,j\in I}\cF(\fU_i\cap\fU_j)\to\dots$. The cohomology of the sheaf $F$ on $X$ associated to $\cF$ is computed by the \v{C}ech complex $\bigoplus_{i\in I}F(U_i)\to \bigoplus_{i,j\in I}F(U_i\cap U_j)\to\dots$ since higher cohomology of coherent sheaves on affinoids vanish. By definition, $F(U_i)=\cF(\fU_i)[\ip]$ so the \v{C}ech complex on the generic fiber is obtained from that of the formal scheme by inverting $p$ and we get the isomorphism $H^i(X,F)\simeq H^i(\fX,\cF)$ which implies the statement by using $\cF=\Omega^j_{\fX/\cO_K}$. 
\end{proof}

\begin{lm}\label{quot}
Let $\fY$ be a smooth proper formal scheme over $\cO_K$ with an action of a finite group $\Gamma$ of order prime to $p$. If the special fiber $\fY_k$ is projective over $k$ and the action of $\Gamma$ on it is free there exists a quotient smooth proper formal scheme $\fX$ over $\cO_K$ with projective special fiber such that $$H^i(\fX,\Omega^j_{\fX/\cO_K})\simeq H^i(\fY,\Omega^j_{\fY/\cO_K})^{\Gamma}$$
\end{lm}

\begin{proof}
For each $n$ the quotient $\fX_n:=\fY_{\cO_K/\fm^n}/\Gamma$ exists by \cite{SGA1} Expose 1 Proposition V.1.8 because admissibility of the action can be checked on the reduced subscheme which is projective in this case. The closed immersions $\fY_{\cO_K/\fm^n}\to \fY_{\cO_K/\fm^{n+1}}$ induce morphisms $\fX_n\to\fX_{n+1}$. Since the action is free, by Corollarie V.2.4 loc. cit. the morphisms $\pi_n:\fY_n\to\fX_n$ are etale with Galois group $\Gamma$ and the canonical maps $\cO_{\fX_n}\otimes_{\bZ}\bZ[\Gamma]\to \pi_{n*}\cO_{Y_n}$ are isomorphisms. 

Hence, the maps $\fX_n\to\fX_{n+1}$ are closed immersions inducing isomorphisms $\fX_{n+1}\times_{\cO_{K}/\fm^{n+1}}\cO_K/\fm^n\simeq \fX_n$ (this also follows from order of $\Gamma$ being invertible on $\fY$ but is not true in general for a non-free action) and the schemes $\fX_n$ form the desired quotient formal scheme $\fX$. If $L$ is an ample line bundle on $\fY_k$ then $\bigotimes\limits_{\gamma\in\Gamma}\gamma^*(L)$ descends to an ample line bundle on $\fX_k$ so $\fX_k$ is projective over $k$.

By the etaleness of the quotient map $\pi:\fY\to\fX$ the canonical morphisms $\Omega^i_{\fX/\cO_K}\to(\pi_*\Omega^i_{\fY/\cO_K})^{\Gamma}$ are isomorphisms and we get a Hochschild-Serre spectral sequence with the second page $E_2^{ab}=H^a(\Gamma, H^b(\fY,\Omega^i_{\fY/\cO_K}))$ converging to $H^{a+b}(\fX,\Omega^i_{\fX/\cO_K})$. Since higher cohomology of $\Gamma$ with coefficients in $\cO_K$-modules vanishes, we get the desired isomorphism.
\end{proof}

\begin{proof}[Proof of Theorem \ref{main}]

We first treat the case of $i+j=3$ and the general case follows from the Lemma \ref{prod} below. Let $\fZ$ be a formal abelian scheme with an action of $G$ over the ring of integers $\cO_L=W(k')$ in some unramified extension $L\supset \bQ_p$  provided by Proposition \ref{formabineq}. We would like to construct $\fX$ by taking the quotient of $\fZ$ by $G$ but first we need to make the action free in a way that does not spoil the discrepancy between ranks of invariant submodules. By \cite{raynaud} Proposition 4.2.3 or \cite{bms} Section 2.2 there exists a smooth projective complete intersection $\fY$ of dimension $\geq 4$ in projective space over $\cO_L$ equipped with a free action of $G$. We equip the product $\fZ\times_{\cO_L}\fY$ with the diagonal action of $G$ which is free because the action on the second factor is free. Note that $H^i(\fZ\times\fY,\cO)\simeq H^i(\fZ,\cO)$ and $H^0(\fZ\times\fY,\Omega^i_{\fZ\times\fY/\cO_L})\simeq H^0(\fZ,\Omega^i_{\fZ/\cO_L})$  for $i\leq 3$ by K\"unneth formula and Lefschetz hyperplane theorem (Proposition 5.3 of \cite{abm}) applied to $\fY$ .

First, we construct the scheme with desired property over the unramified extension $\cO_L$ as the quotient $\fX':=(\fZ\times\fY)/G$ provided by the Lemma \ref{quot}. We have $H^3(\fX',\cO)=H^3(\fZ\times\fY,\cO)^G=H^3(\fZ,\cO)^G$ and $H^0(\fX',\Omega^3_{\fX'/\cO_L})^G=H^0(\fZ,\Omega^3_{\fZ/\cO_L})^G$.  Proposition \ref{formabineq} supplied $\fZ$ such that the ranks of invariant modules $H^3(\fZ,\cO)^G$ and $H^0(\fZ,\Omega^3_{\fZ/\cO_L})^G$ are different, so $H^3(\fX',\cO)$ and $H^0(\fX',\Omega_{\fX'/\cO_L}^3)$ are indeed of different rank. By Proposition \ref{deg2} we have $3h^{3,0}(\fX')+h^{2,1}(\fX')=3h^{0,3}(\fX')+h^{1,2}(\fX')$ so the ranks of $H^2(\fX',\Omega^1_{\fX'/\cO_L})$ and $H^1(\fX',\Omega^2_{\fX'/\cO_L})$ are forced to be different as well. 

Finally, the example over $\bZ_p$ is obtained by Weil restriction of scalars $\fX:=\Res^{\cO_L}_{\bZ_p}\fX'$. Since $\cO_L$ is finite etale over $\bZ_p$ this is a smooth proper formal scheme over $\bZ_p$ with projective special fiber $\fX_{\bF_p}=\Res^{k'}_{\bF_p}\fX_{k'}$. Since $\fX\times_{\bZ_p}\cO_L$ is isomorphic to $(\fX')^{\times d}$ where $d$ is the degree of $k'$ over $\bF_p$, the Hodge cohomology modules of $\fX$ are free and $h^{3,0}(\fX)-h^{0,3}(\fX)=d(h^{3,0}(\fX')-h^{0,3}(\fX'))$ because the symmetry on cohomology of $\fX'$ in degrees $1$ and $2$ holds by Corollary \ref{cordeg2}.

\begin{lm}\label{prod}
Let $\fX$ be a smooth proper formal scheme over $\cO_K$ such that $h^{3,0}(\fX)\neq h^{0,3}(\fX)$. Then for any $i,j\geq 0$ with $i\neq j$ and $i+j>3$ there exists a smooth projective scheme $\fY$ over $\cO_K$ such that $h^{i,j}(\fX\times\fY)\neq h^{j,i}(\fX\times \fY)$.
\end{lm}

\begin{proof}
Recall that for a formal scheme $\fT$ the number $\delta^{i,j}(\fT)$ is defined as the difference $h^{i,j}(\fT)-h^{j,i}(\fT)$. Note that if $\fY$ satisfies Hodge symmetry then \begin{equation}\label{productdisc}\delta^{i,j}(\fT\times\fY)=\displaystyle\sum\limits_{\substack{i_1+i_2=i \\ j_1+j_2=j}}\delta^{i_1,j_1}(\fT)h^{i_2,j_2}(\fY)\end{equation}

By Proposition \ref{deg2} we have $\delta^{2,1}(\fX)=-3\delta^{3,0}(\fX)$ and $\delta^{1,0}(\fX)=\delta^{2,0}(\fX)=0$. Let us assume that $i>j$. For a scheme $\fY$ denote by $H_{\fY}(x,y):=\sum\limits_{i,j}h^{i,j}(\fY)x^iy^j$ its Hodge polynomial. If $\fY_1\subset\fY_2$ is a smooth closed subscheme of codimension $r+1$ in a smooth proper scheme then the Hodge polynomial of the blow-up is given by $$H_{Bl_{\fY_1}\fY_2}(x,y)=H_{\fY_2}(x,y)+H_{\fY_1}(x,y)\cdot(xy+(xy)^2+\dots+(xy)^r)$$ Let $\fT_{d,n}$ be a smooth degree $d$ hypersurface in $\bP^{n+1}_{\cO_K}$ where $d$ and $n$ are numbers that we will choose later. Construct a sequence of smooth projective schemes starting with $\fY_0=\fT_{d,n}$ and for $s\geq 0$ defining the scheme $\fY_{s+1}$ as the blow-up of some projective space $\bP^{N_s}_{\cO_K}$ along some embedding $\fY_s\subset\bP^{N_s}_{\cO_K}$ of codimension $\geq 2$ (the numbers $N_s>\dim_{\cO_K}\fY_s+1$ can be chosen arbitrarily). By the above formula we have $H_{\fY_s}(x,y)=F(xy)+H_{\fT_{d,n}}(x,y)\cdot (xy)^s\cdot G(xy)$ where $F(t),G(t)\in 1+t\bZ[t]$ are some polynomials.

Assume first that $i> j+3$. Then take $n=i-3-j,s=j$ and consider $\fY:=\fY_s$ obtained by the above inductive procedure from $\fT_{d,n}$. We then have $h^{i-2,j-1}(\fY)=h^{i-1,j-2}(\fY)=h^{i,j-3}(\fY)=0$ and $h^{i-3,j}(\fY)=h^{n,0}(\fT_{d,n})$. This turns the formula (\ref{productdisc}) into $$\delta^{i,j}(\fX\times\fY)=\delta^{3,0}(\fX)\cdot h^{i-3,j}(\fY)+\sum\limits_{i_2+j_2<i+j-3}\delta^{i-i_2,j-j_2}(\fX)\cdot h^{i_2,j_2}(\fY)$$ By \cite{SGA72} XI.Theoreme 1.5 the Hodge numbers $h^{a,b}(\fT_{d,n})$ of a hypersurface are zero unless $a+b=n$ or $a=b$ and $h^{a,a}(\fT_{d,n})=1$ for $0\leq a\leq n, 2a\neq n$. It follows that the Hodge polynomial $H_{\fY}(x,y)$ is congruent modulo the ideal $(x,y)^{n+s}$ to a polynomial of the form $K(xy)$ with coefficients not depending on the degree $d$ of the hypersurface. In particular, $\delta^{i,j}(\fX\times\fY)$ is the sum of $\delta^{3,0}(\fX)\cdot h^{i-3,j}(\fY)=\delta^{3,0}(\fX)\cdot h^{n,0}(\fT_{d,n})$ and a number that does not depend on $d$. By loc. cit. Corollaire 2.4 the number is $h^{n,0}(\fT_{d,n})$ is equal to $\binom{d-1}{n+1}$ so for large enough $d$ this sum will be non-zero.

If $i=j+2$ then take $\fY=\fY_s$ with $n=1,s=j-1$. An analogous computation gives that $\delta^{i,j}(\fX\times\fY)$ is the sum of a number that does not depend on $d$ and the expression $\delta^{2,1}(\fX)\cdot h^{1,0}(\fT_{d,1})+\delta^{3,0}(\fX)\cdot h^{0,1}(\fT_{d,1})=\delta^{3,0}(\fX)(h^{0,1}(\fT_{d,1})-3h^{1,0}(\fT_{d,1}))=-2\delta^{3,0}(\fX)\binom{d-1}{2}$ which is non-zero for $d>2$.

Finally, if $i= j+3$ or $i=j+1$ we will be able to move asymmetry into higher cohomology just via Tate twists. Consider the scheme $\fY=(\bP^1_{\cO_K})^d$ where we will again choose $d$ at the end. We have $$\delta^{i,j}(\fX\times\fY)=\sum\limits_{r\leq \frac{i+j-3}{2}}\delta^{i-r,j-r}(\fX)\cdot\binom{d}{r}$$ This is a polynomial in $d$ with leading coefficient $\delta^{i-r,j-r}(\fX)$ with $r=\frac{i+j-3}{2}$ which is nothing but $\pm\delta^{3,0}(\fX)$ or $\pm\delta^{2,1}(\fX)$ so this expression is non-zero for a large enough value of $d$.
\end{proof}

This finishes the proof of Theorem \ref{main}. \end{proof}

\begin{cor}\label{maincor}
For every pair of positive numbers $i\neq j$ with $i+j\geq 3$ there exists a smooth proper rigid-analytic variety $X$ over $\bQ_p$ admitting a smooth formal model $\fX$ over $\Spf\bZ_p$ with projective special fiber $\fX_{\bF_p}$ such that $$h^{i,j}(X)\neq h^{j,i}(X).$$
\end{cor}

\begin{proof}
Take $X$ to be the generic fiber of a formal scheme $\fX$ provided by Theorem \ref{main}. Applying Lemma \ref{gencoh} gives the result.
\end{proof}

\section{Embellishing the main example and a question}

We end by noting that the examples provided by Theorem \ref{main} can be modified to avoid some of the characteristic $p$ ``pathologies".

In all of the examples constructed above the Hodge cohomology groups $H^i(\fX,\Omega^j_{\fX/\bZ_p})$ are free modules over $\bZ_p$ so the Hodge numbers of the special and generic fibers of $\fX$ coincide. In particular, the Hodge symmetry for the special fiber $\fX_{\bF_p}$ fails as well. The next lemma shows that symmetry for the generic fiber cannot be salvaged by requiring that for the special fiber: there are enough projective schemes failing symmetry for the special fiber to cancel out the asymmetry on the special fiber by taking products. We only treat here one particular pair of degrees. 

\begin{lm}\label{specialfibersym}
There exists a smooth proper formal scheme $\fX'$ over $\Spf\bZ_p$ with projective special fiber such that $h^{3,0}(\fX')\neq h^{0,3}(\fX')$ while $h^{3,0}(\fX'_{\bF_p})=h^{0,3}(\fX'_{\bF_p})$.
\end{lm}

\begin{proof}
The proof of Theorem \ref{main} provides us with a smooth proper scheme $\fX$ over $\bZ_p$ with projective special fiber such that all Hodge groups $H^i(\fX,\Omega^j_{\fX/\bZ_p})$ are free over $\bZ_p$ and $h^{3,0}(\fX)< h^{0,3}(\fX)$ while $h^{1,0}(\fX)=h^{0,1}(\fX)=0$ (the Hodge numbers in degree $1$ vanish because representations $U,V$ used in proof of Proposition \ref{formabineq} have trivial invariants). Since Hodge cohomology modules are free, the Hodge numbers of the special fiber $\fX_{\bF_p}$ are equal to those of $\fX$.

We will now take the product of $\fX$ with an auxiliary projective scheme that will be constructed by approximating the stack $B(\mu_p\times\bZ/p)$. By Theorem 1.2 of \cite{abm} applied over $\bZ_p$ to the group scheme $\mu_p\times\bZ/p$ with $d=3$ there exists a smooth projective scheme $\fY'$ such that $H^j(\fY'_{\bF_p},\Omega^i_{\fY'_{\bF_p}})=H^j(B(\mu_p\times\bZ/p)_{\bF_p},\Omega^i)$ and $H^j(\fY',\Omega^i_{\fY'/\bZ_p})=H^j(B(\mu_p\times\bZ/p)_{\bZ_p},\Omega^i)$ for pairs $i,j$ with $i+j\geq 3$ and $i=0$ or $j=0$. The Hodge polynomial $H_{B(\mu_p\times\bZ/p)_{\bF_p}}(x,y):=\sum\limits_{i,j}h^{i,j}(B(\mu_p\times\bZ/p)_{\bF_p})x^iy^j$ of this stack is given by the product $H_{B\mu_{p,\bF_p}}(x,y)H_{B\bZ/p_{\bF_p}}(x,y)$. By \cite{totaro} Proposition 11.1 the first multiple $H_{B\mu_{p,\bF_p}}$ is equal to $\frac{1+x}{1-xy}$ and $H_{B(\bZ/p)_{\bF_p}}$ is equal to $\frac{1}{1-y}$ because the only non-zero Hodge cohomology groups of the classifying stack of a finite discrete group $\Gamma$ are given by group cohomology $H^i(B\Gamma_{\bF_p},\cO)\simeq H^i(\Gamma,\bF_p)$. Take $\fY:=\fY'\times \fE^l$ where $\fE$ is an elliptic curve over $\bZ_p$ and $l$ is a number that we will choose at the end.

By the above computation, $H_{\fY_{\bF_p}}(x,y)$ is given by $\frac{1}{1-y}\cdot\frac{1+x}{1-xy}\cdot(1+x+y+xy)^l$ modulo an element of the ideal $(x^4,xy,y^4)$. For the purpose of computing $h^{3,0}((\fX\times\fY)_{\bF_p})$ and $h^{0,3}((\fX\times\fY)_{\bF_p})$ we only care about this polynomial modulo $xy$ and we have $H_{(\fX\times\fY)_{\bF_p}}(x,y)\equiv (1+h^{2,0}(\fX_{\bF_p})x^2+h^{3,0}(\fX_{\bF_p})x^3+h^{0,2}(\fX_{\bF_p})y^2+h^{0,3}(\fX_{\bF_p})y^3)(1+(l+1)x+(\binom{l}{2}+l)x^2+(\binom{l}{3}+\binom{l}{2})x^3+(l+1)y+(\binom{l}{2}+l+1)y^2+(\binom{l}{3}+\binom{l}{2}+l+1)y^3)$ modulo the ideal $(x^4,xy,y^4)$. The Hodge numbers of the special fiber $\fX_{\bF_p}\times\fY_{\bF_p}$ are thus given by $h^{3,0}((\fX\times\fY)_{\bF_p})=\binom{l}{3}+\binom{l}{2}+h^{3,0}(\fX_{\bF_p})+(l+1)h^{2,0}(\fX_{\bF_p})$, $h^{0,3}((\fX\times\fY)_{\bF_p})=\binom{l}{3}+\binom{l}{2}+l+1+h^{0,3}(\fX_{\bF_p})+(l+1)h^{0,2}(\fX_{\bF_p})$. Hence, $\delta^{3,0}((\fX\times\fY)_{\bF_p})=\delta^{3,0}(\fX_{\bF_p})+l+1$. Since $\delta^{3,0}(\fX_{\bF_p})<0$ there exists $l$ that makes this difference equal to zero. 

The product $\fX\times\fY$ is our desired scheme $\fX'$. Note that the Hodge numbers $h^{i,0}(\fY_{\bQ_p})$ of the generic fiber $\fY_{\bQ_p}$ vanish for $i\leq 3$ so the asymmetry $h^{3,0}(\fX')=h^{3,0}(\fX)\neq h^{0,3}(\fX)= h^{0,3}(\fX')$ persists.
\end{proof}

Another feature of the counterexamples to Hodge symmetry obtained by taking quotients of formal abelian schemes by finite groups is the absence on the special fiber of an ample line bundle of degree prime to $p$. Indeed, if the $d$-dimensional special fiber of $\fX$ (in the notation of the proof of Theorem \ref{main}) admits a line bundle $L$ such that $L^d\not\equiv 0\pmod{p}$ then it induces a $G$-equivariant ample line bundle $L'$ on $\fZ_{k'}$ with $(L')^{\dim \fZ_{k'}}\not\equiv 0\pmod{p}$ as well (because the order of $G$ is prime to $p$). It induces a separable $G$-equivariant polarization and hence and isomorphism $H^0(\fZ,\Omega^1_{\fZ_k'/k'})\simeq H^1(\fZ_{k'},\cO)^{\vee}$ of $G$-representations. Hence, for all $i$ and $j$ the representations $H^i(\fZ_{k'},\Omega^j_{\fZ_{k'}/k'})$ and $H^j(\fZ_{k'},\Omega^i_{\fZ_{k'}/{k'}})$ would be dual to each other and Hodge symmetry for $\fX_{k'}$ and, hence, for $\fX$ would have to hold.

We note, however, that blowing up a point provides the special fiber with a prime to $p$ polarization without spoiling Hodge asymmetry

\begin{lm}
For every pair of distinct positive integers $i,j$ with $i+j\geq 3$ there exists a smooth proper formal scheme $\fX'$ over $\bZ_p$ such that $h^{i,j}(\fX')\neq h^{j,i}(\fX')$ and the special fiber carries an ample divisor $L$ such that its top self-intersection is prime to $p$: $L^{\dim(\fX'_{\bF_p})}\not\equiv 0\pmod{p}$.
\end{lm}

\begin{proof}
Let $\fX$ be a formal scheme provided by Theorem \ref{main}. Choose a point $x:\Spf\bZ_p\to\fX$ (one sees easily that the proof of Theorem \ref{main} can be modified to make sure that at least one point exists). Let $f:\fX'\to\fX$ be the blow-up of $\fX$ at the closed subscheme given by $x$. Note that $h^{i,j}(\fX')=h^{i,j}(\fX)$ for $i\neq j$.

The special fiber $\fX'_{\bF_p}$ is the blow-up of $\fX_{\bF_p}$ at a point, denote by $E\subset \fX'_{\bF_p}$ the exceptional divisor. Let $H$ be any ample divisor on $\fX'_{\bF_p}$. We will find a linear combination $mH+nE$ with $n\geq 0$ such that the top self-intersection $(mH+nE)^d$ is prime to $p$.

Suppose that $H|_E$ has degree $k$ in $\Pic(E)=\bZ$. Then $H^{d-i}\cdot E^i=k^{d-i}\cdot(-1)^i$ for $i>0$ by Example 8.3.9. in \cite{fulton} and the projection formula, so \begin{equation}\label{blowup degree}(H+nE)^d=H^d+\sum\limits_{i=1}^d\binom{d}{i}n^ik^{d-i}(-1)^{i-1}=H^d-(k-n)^d+k^d\end{equation} For varying $n$ the residue of $(k-n)^d$ modulo $p$ takes at least two different values, so we can find $n$ such that the expression (\ref{blowup degree}) is prime to $p$. For a large enough $r$ the divisor $(pr+1)H+nE$ is ample and has non-zero self-intersection modulo $p$, as desired.
\end{proof}

Apart from satisfying Hodge symmetry, compact K\"ahler manifolds must have nonzero middle Hodge numbers $h^{i,i}(X)\neq 0$ for $i\leq\dim X$ because $H^i(X,\Omega^{i}_{X/\bC})$ contains the nonzero $i$-th power of the class of a K\"ahler form. It would be interesting to find out whether this is the case for rigid-analytic varieties with projective reduction:

\begin{quest} Does there exist a smooth proper rigid-analytic variety $X$ of $\dim X>0$ admitting a formal model with projective special fiber such that $H^1(X,\Omega^1_{X/K})=0$?
\end{quest}

Note that $X$ must have non-zero second de Rham cohomology $H^2_{\dR}(X/K)$. If $X$ is algebraic projective, then a non-zero class is given by the first Chern class of an ample line bundle. If, on the contrary, $X$ is not projective, then no ample line bundle on the special fiber $\fX_k$ can lift to $\fX$. The obstruction to lifting a line bundle gives a non-zero class in second cohomology as follows. Consider the long exact sequence $$\dots\to \Pic(\fX)\to \Pic(\fX_k)\xrightarrow{\delta} H^2(\fX,(1+\fm\cO_{\fX})^{\times})\to\dots$$ obtained from the short exact sequence $1\to (1+\fm\cO_{\fX})^{\times}\to\cO_{\fX}^{\times}\to\cO_{\fX_k}^{\times}\to 1$ of sheaves of abelian groups on the Zariski site of $\fX$. For any ample line bundle $L\in \Pic(\fX_k)$ the image $\delta(L)\in H^2(\fX,(1+\fm\cO_{\fX})^{\times})$ is non-zero and, since no power of $L$ lifts to $\fX$, is a non-torsion element. Applying to $\delta(L)$ the isomorphism given by the logarithm $\log:(1+\fm\cO_{\fX})^{\times}[\ip]\simeq \cO_{\fX}[\ip]$ gives a non-zero class in $H^2(\fX,\cO_{\fX}[\ip])=H^2(X,\cO)$.

\bibliographystyle{alpha}
\bibliography{hsym_v6}

\end{document}